\documentclass[a4paper,twoside,10pt]{article}

\usepackage{amsmath, mathrsfs}
\numberwithin{equation}{section}

\usepackage{mathtools}

\usepackage{amsthm}
\usepackage{aliascnt}

\newtheorem{theorem}{Theorem}[section]

\newaliascnt{lemma}{theorem}
\newtheorem{lemma}[lemma]{Lemma}
\aliascntresetthe{lemma}

\newaliascnt{proposition}{theorem}
\newtheorem{proposition}[proposition]{Proposition}
\aliascntresetthe{proposition}

\newaliascnt{corollary}{theorem}
\newtheorem{corollary}[corollary]{Corollary}
\aliascntresetthe{corollary}

\theoremstyle{definition}
\newaliascnt{definition}{theorem}
\newtheorem{definition}[definition]{Definition}
\aliascntresetthe{definition}

\theoremstyle{remark}
\newaliascnt{remark}{theorem}
\newtheorem{remark}[remark]{Remark}
\aliascntresetthe{remark}

\usepackage{bbm}
\newcommand{\N}{\mathbbm N}

\newcommand{\R}{\mathbbm R}

\usepackage[a4paper,centering,bindingoffset=0cm,marginpar=0cm]{geometry}

\usepackage{fancyhdr}
\usepackage{titlesec}
\titleformat{\section}{\filcenter\sc\large}{\thesection.\;}{0em}{}
\titleformat{\subsection}[runin]{\bf}{\thesubsection.\;}{0em}{}[.]

\usepackage[backend=bibtex,maxnames=10,firstinits=true]{biblatex}
\DeclareFieldFormat[article,inbook,incollection,inproceedings,patent,thesis,unpublished]{title}{{\em#1}}
\DeclareFieldFormat[article,inbook,incollection,inproceedings,patent,thesis,unpublished]{journaltitle}{#1}
\DeclareFieldFormat[incollection,inproceedings]{booktitle}{#1}
\AtEveryBibitem{%
  \clearname{translator}%
  \clearfield{pagetotal}%
  \clearlist{language}%
  \clearfield{month}
}
\renewbibmacro{in:}{}
\DeclareFieldFormat{eprint:euclid}{Project euclid: \href{http://projecteuclid.org/euclid.\thefield{eprint}}{\thefield{eprint}}}

\usepackage[pdftex,colorlinks=true,linkcolor=blue,citecolor=red,bookmarks=true,bookmarksnumbered=true]{hyperref}
\hypersetup{%
	pdfauthor={Franziska Borer}, 
	pdfsubject={Uniqueness of the Ricci flow}, 
	pdftitle={Uniqueness of Weak Solutions for the Normalised Ricci Flow in Two Dimensions}, 
	pdfkeywords={Uniqueness, Ricci flow, Weak solution, Evolution Equation}}

\def\XXint#1#2#3{{\setbox0=\hbox{$#1{#2#3}{\int}$}
     \vcenter{\hbox{$#2#3$}}\kern-.5\wd0}}

\newcommand{\e}{\mathrm{e}}

\DeclareMathOperator{\vol}{\mathrm{vol}}
\DeclareMathOperator{\esssup}{\mathrm{esssup}}

\title{Uniqueness of Weak Solutions for the Normalised Ricci Flow in Two Dimensions}
\author{Franziska Borer\thanks{ETH Zurich, Department of Mathematics, R\"amistrasse 101, 8092 Zurich, Switzerland\newline email: \href{mailto:franziska.borer@math.ethz.ch}{franziska.borer@math.ethz.ch}}}
\date{}

\begin{document}
\maketitle

\begin{abstract}
We show uniqueness of classical solutions of the normalised two-dimensional Hamilton-Ricci flow on closed, smooth manifolds for smooth data among solutions satisfying (essentially) only a uniform bound for the Liouville energy and a natural space-time $L^2$-bound for the time derivative of the solution. The result is surprising when compared with results for the harmonic map heat flow, where non-uniqueness through reverse bubbling may occur.
\end{abstract}

\section{Introduction}
We consider Hamilton's normalised Ricci flow on a two-dimensional, smooth, connected, closed Riemannian surface $(M,g_0)$.
Hamilton \cite{Ham88} showed the global existence and uniqueness of this flow for smooth initial metrics $g_0$.

Here, we want to investigate the uniqueness of Hamilton's solution also among suitable  weak solutions of the flow.
The energy identity for classical solutions of the normalised Ricci flow will help us  characterise the natural space of admissible weak solutions; see \autoref{adSol} below. Comparing with similar results for the harmonic map heat flow (see \cite{Rup08}), we surprisingly do not have to assume that the energy of the weak solution is (essentially) decreasing.

We will show that any admissible weak solution of the normalised Ricci flow for given smooth initial data coincides with the classical solution of Hamilton.
In fact, our results not only hold for smooth data but also for data of class $H^2$. This is by no means trivial as we can see in the case of the harmonic map heat flow. There, in a corresponding class of weak solutions, non-uniqueness through reverse bubbling can occur, as shown by Topping \cite{Top02HMF} and by Bertsch, Dal Passo and Van der Hout \cite{BerPassHou02}. In this case uniqueness only holds if upward jumps of the energy are a-priori restricted to size smaller than $4\pi$, as shown by Ruplin \cite{Rup08}, after a conjecture by Topping \cite{Top96_thesis}.

This paper is structured as follows: First we review some properties of classical solutions of the normalised Ricci flow. This will motivate our definition of an admissible weak solution.
Writing $g=\e^{2u}\bar g$ for our evolving metrics $g=g(t)$, the main part of the paper then consists in showing that the difference of a weak solution $v$ and the classical solution $u$ with the same initial data is of constant sign.
A measure theoretic argument together with conservation of volume then shows that the two solutions are equal almost everywhere, which gives our main result. In \cite{GieTop10}, Giesen and Topping used a similar argument to show that the unnormalised Ricci flow on surfaces has a unique, global solution for incomplete initial metrics $g_0$ with $K_{g_0}\le -\eta<0$ within the class of instantaneously complete Ricci flows.

\textbf{Acknowledgement:} I cordially thank Michael Struwe for his helpful mentoring. This work was supported by the Swiss National Science Foundation project ``Curvature and Criticality in Geometric Analysis'' (project number \href{http://p3.snf.ch/Project-140467}{140467}).

\section{Classical Solutions for the Normalised Ricci Flow in Two Dimensions}\label{CS}
Let $(M,g_0)$ be a smooth, two-dimensional, closed (that is compact without boundary), connected Riemannian manifold. The normalised Ricci flow (introduced by Hamilton \cite{Ham82} in 1982) deforms in two dimensions the metric $g_0$ under the evolution equation
\begin{equation}\label{normalisierterRF2dim}
\begin{aligned}
\partial_tg(t)&=(r_{g(t)}-2K_{g(t)})g(t),\quad t>0;\\
g(0)&=g_0,
\end{aligned}
\end{equation}
where $K_{g(t)}$ denotes the Gauss curvature of the Riemannian metric $g(t)$ and
\[r_{g(t)}=\frac{2}{\vol_{g(t)}}\int_MK_{g(t)}d\mu_{g(t)}.\]
Here, $\vol_{g(t)}=\int_Md\mu_{g(t)}$ denotes the volume of the manifold with respect to the metric $g(t)$.
The term $r_{g(t)}g(t)$ in \eqref{normalisierterRF2dim} ensures that this volume remains constant: Indeed, we have
\[ \frac d{dt}\vol_{g(t)} = \frac d{dt}\int_Md\mu_{g(t)} = \int_M(r_{g(t)}-2K_{g(t)})d\mu_{g(t)} = 0. \]
By the uniformisation theorem (see e.g.~\cite[Theorem 1.7, page 7]{Lee97}) there exists a metric $\bar g$, we call it background metric, which is conformal to $g_0$ and has constant curvature. This means that $g_0$ can be written as $g_0=\e^{2u_0}\bar g$ for a suitable function $u_0$, and $K_{\bar g}\equiv \bar K\in\R$, where $K_{\bar g}$ denotes the Gauss curvature of $\bar g$.
Considering equation \eqref{normalisierterRF2dim} we see that the change in the metric is pointwise a multiple of the metric. So, the conformal class for an initial metric $g_0=\e^{2u_0}\bar g$ is preserved.
Therefore we may express the solution by $g(t)=\e^{2u(t)}\bar g$ with $u(0)=u_0$. Using now that according to the Gauss--Bonnet theorem
\[r_{g(t)}=2\frac{\vol_{\bar g}}{\vol_{g(t)}}\bar K=2\frac{\vol_{\bar g}}{\vol_{g_0}}\bar K\]
(and observing that the scalar curvature $R_{g(t)}$ in two dimension is just twice the Gauss curvature $K_{g(t)}$), the equation \eqref{normalisierterRF2dim} reads
\begin{equation}\label{konformK}
\begin{aligned}
\partial_tu(t)&=\frac{\vol_{\bar g}}{\vol_{g_0}}\bar K-K_{g(t)}, \quad t>0;\\
u(0)&=u_0.
\end{aligned}
\end{equation}
Without loss of generality we may assume $\vol_{\bar g}=\vol_{g_0}$. So by using the Gauss equation
\begin{equation}\label{Gauss}
K_{g(t)} = \e^{-2u(t)}(\bar K-\Delta_{\bar g}u(t))
\end{equation}
to calculate the Gau{\ss} curvature $K_{g(t)}$ of the metric $g(t)$, the normalised Ricci flow equation~\eqref{normalisierterRF2dim} reduces to
\begin{equation}\label{konformRF}
\begin{aligned}
\partial_tu(t)&=\e^{-2u(t)}\Delta_{\bar g}u(t)+\bar K(1-\e^{-2u(t)}),\quad t>0;\\
u(0)&=u_0.
\end{aligned}
\end{equation}

Hamilton \cite{Ham88} and Chow \cite{Cho91} showed that for every smooth, compact surface $M$ without boundary and every smooth initial metric $g_0$ on $M$, there exists a unique, smooth, global solution $g(t)$ of \eqref{normalisierterRF2dim} which as $t\to\infty$ converges, exponentially fast, to a metric of constant curvature.

Let
\[U_T:=L^\infty((0,T); H^2(M,\bar g))\cap L^2((0,T);H^3(M,\bar g))\cap H^1((0,T);H^1(M,\bar g));\]
see the next section for a precise definition.
In \cite{Str02}, Struwe showed the existence of a unique, global solution $u$ of \eqref{konformRF} in the space $U_T$ for data $u_0\in H^2(M,\bar g)$, which is classical for $t>0$, and was able to give a simpler proof of exponentially fast convergence.

\begin{theorem}[Struwe, \cite{Str02}]\label{Struwe}
For any $u_0\in H^2(M,\bar g)$ there exists a unique, global solution $u\in U_T$ of \eqref{konformRF} which is smooth for $t>0$ and preserves volume.
\end{theorem}

For (classical) solutions of \eqref{konformRF} upon testing equation \eqref{konformRF} with $u_t\e^{2u}$ we see that the Liouville energy
\begin{equation}\label{LiEn}
E(u(t)):=\frac12\int_M(|\nabla_{\bar g}u(t)|^2_{\bar g}+2\bar Ku(t))d\mu_{\bar g}
\end{equation}
of $u(t)$ is decreasing in time and there holds the energy identity
\begin{equation}\label{energie}
\begin{split}
-\infty<E(u(T))&=E(u_0)-\int_0^T\int_M\e^{2u(t)}|\partial_tu(t)|^2d\mu_{\bar g}dt\\
&=E(u_0)-\int_0^T\int_M|\bar K-K_{g(t)}|^2d\mu_gdt\le E(u_0)<\infty,
\end{split}
\end{equation}
for all $0\le T<\infty$. 
These observations will help us find suitable conditions to guarantee uniqueness for weak solutions of the normalised Ricci flow as we will see in the next section.

\section{Weak Solutions for the Normalised Ricci Flow in Two Dimensions}
In the setting of \autoref{CS} we now choose for simplicity $\vol_{\bar g}=1=\vol_{g_0}$. In the following we call the unique, global solution $u$ of \eqref{konformRF} for data $u_0\in H^2(M,\bar g)$ provided by \autoref{Struwe} the {\em reference solution}. Since the reference solution preserves the volume we have $\vol(M,g(t))=1$ for all $t$, where $g(t)=\e^{2u(t)}\bar g$.
For a given $T>0$ and $p\in[1,\infty)$, $q\in[1,\infty]$ we write $L^p_tL^q_x$ for the space $L^p([0,T];L^q(M,\bar g))$ with the norm
\[\|u\|_{L^p_tL^q_x}:=\left(\int_0^T\|u(t)\|^p_{L^q(M,\bar g)} dt\right)^{\frac1p},\]
and analogously define $L^p_tH^k_x$, $k\in\N$. Similarly, we denote the space $L^\infty([0,T];L^p(M,\bar g))$ by $L^\infty_tL^p_x$, $p\in[1,\infty]$, with
\[\|u\|_{L^\infty_tL^p_x}:=\esssup_{t\in[0,T]}\|u(t)\|_{L^p(M,\bar g)}\]
and analogously introduce $L^\infty_tH^k_x$ for $k\in\N$.
Furthermore, we use the abbreviation $H^1_tH^1_x$ for the space $H^1((0,T);H^1(M,\bar g))$ with
\[\|u\|_{H^1_tH^1_x}:=\left(\int_0^T\int_M(|u(t)|^2+|\partial_tu(t)|^2+|\nabla_{\bar g}u(t)|^2_{\bar g}+|\nabla_{\bar g}\partial_tu(t)|^2)d\mu_{\bar g}dt\right)^{\frac12}.\]

Finally we set
\[U_T:=L^\infty_tH^2_x\cap L^2_tH^3_x\cap H^1_tH^1_x\]
and
\[V_T:=L^\infty_tH^1_x.\]

Now, we define the class of suitable weak solutions of equation \eqref{konformRF} for initial conditions $u_0\in H^2(M,\bar g)$.
By \eqref{energie} it is natural to require also for a weak solution $v$ of the normalised Ricci flow that the Liouville energy is uniformly bounded along $v$. We therefore impose the condition $v\in V_T$.

Moreover, we require that there exists a weak time derivative $\partial_t v\in L^2((0,T);L^2(M,\e^{2v}\bar g))$, i.e.
\begin{equation}\label{energieV}
\int_0^T\int_M \e^{2v(t)}|\partial_tv(t)|^2_{\bar g}d\mu_{\bar g}dt\le C<\infty.
\end{equation}

\begin{remark}
In fact, integrability of $\partial_tv\in L^1_tL^1_x$ with respect to $d\mu_{\bar g}dt$ follows immediately from \eqref{energieV} and \autoref{Reg_w}, see \autoref{RR}. 
Since $T<\infty$, we additionally have $v\in L^\infty_tH^1_x\subset L^1_tW^{1,1}_x$.
\end{remark}

\begin{remark}\label{Randwerte}
We can identify $\tilde v(t,x)=v(t)(x)$ and see that $\tilde v\in W^{1,1}((0,T)\times M)$ since
\begin{align*}
\|\tilde v(t,x)\|_{W^{1,1}}
&=\int_0^T\int_M(|\tilde v(t,x)|+|\partial_t\tilde v(t,x)|+|\nabla_{\bar g}\tilde v(t,x)|)d\mu_{\bar g}(x)dt
<\infty.
\end{align*}
Therefore by the trace theorem
$\tilde v(0,\cdot)= v(0)$ and $\tilde v(T,\cdot)=v(T)$ exist in the sense of $L^1$-traces.
\end{remark}

\begin{definition}\label{adSol}
Let $u_0\in H^2(M,\bar g)$. We call a function $v\in V_T$ an \textit{admissible weak solution of the Ricci flow \eqref{konformRF}} with initial data $u_0$ if there exists a weak time derivative $\partial_tv\in L^2((0,T);L^2(M,\e^{2v(t)}\bar g))$ of $v$, $v(0)=u_0$,
and if there holds
\begin{multline}\label{weaksol}
\int_0^T\int_M\partial_tv(t)\e^{2v(t)}\varphi(t)d\mu_{\bar g}dt\\
=-\int_0^T\int_M\left(\langle\nabla_{\bar g}v(t),\nabla_{\bar g}\varphi(t)\rangle_{\bar g}-\bar K(\e^{2v(t)}-1)\varphi(t)\right)d\mu_{\bar g}dt
\end{multline}
for any test function $\varphi\in C^\infty_c((0,T);C^\infty(M,\bar g))$.
\end{definition}

Now we can state our main result.
\begin{theorem}\label{HT}
Given $u_0\in H^2(M,\bar g)$, let $u\in U_T$ be the reference solution of \eqref{konformRF} provided by \autoref{Struwe}.
Furthermore let $v\in V_T$ be an admissible weak solution for the Ricci flow \eqref{konformRF} with initial data $u_0$.
Then we have $u\equiv v$ almost everywhere.
\end{theorem}

It would be interesting if a similar uniqueness result holds within the class of weak solutions for initial data $u_0\in H^1(M,\bar g)$.

An important fact about admissible weak solutions of the normalised Ricci flow is the conservation of volume.
\begin{lemma}[Conservation of volume]\label{VolumenerhaltungV}
If $v\in V_T$ is an admissible weak solution of the Ricci flow \eqref{konformRF} with initial data $u_0$ then
\[\int_M\e^{2v(t)}d\mu_{\bar g}\equiv\vol_{g_0}=1\quad\text{for almost all}\quad t\in[0,T].\]
\end{lemma}

\begin{proof}
Let $h(t)=\e^{2v(t)}\bar g$, $q(t):=\vol_{h(t)}-1$ and $\tilde\varphi\in C^\infty_c((0,T))$. Since $\tilde\varphi$ is independent of $x\in (M,\bar g)$ we get by using $\vol_{\bar g}=1$ that
\begin{align*}
-\int_0^Tq(t)\tilde\varphi'(t)dt&=-\int_0^T\left(\int_M(\e^{2v(t,x)}-1)d\mu_{\bar g}(x)\right)\tilde\varphi'(t)dt\\
&=2\int_0^T\int_M \partial_tv(t,x)\e^{2v(t,x)}\tilde \varphi(t)d\mu_{\bar g}(x)dt.\\
\end{align*}
Since $v$ is a weak solution of the Ricci flow \eqref{konformRF}, it fulfils by \autoref{adSol} the relation~\eqref{weaksol} for all test functions $\varphi\in C^\infty_c((0,T);C^\infty(M,\bar g))$. Evaluating this for the test function $\varphi$ given by $\varphi(t,x)\equiv\tilde\varphi(t)$, we obtain
\begin{align*}
\int_0^T\int_M \partial_tv(t,x)\e^{2v(t,x)}\tilde\varphi(t)d\mu_{\bar g}(x)dt&=\int_0^T\int_M\bar K(\e^{2v(t,x)}-1)\tilde\varphi(t)d\mu_{\bar g}(x)dt\\
&=\bar K\int_0^tq(t)\tilde\varphi(t)dt.
\end{align*}

We write this in the form
\[\int_0^T(\partial_tq(t)-2\bar Kq(t))\tilde\varphi(t)dt=0,\quad\text{for all}\quad \tilde\varphi\in C^\infty_c((0,T)).\]
Hence $q$ solves
\[\partial_t q(t)=2\bar K q(t)\]
with $q(0)=0$, and $q\equiv0$,
which concludes the proof.
\end{proof}

\section{Proof of \autoref{HT}}\label{proof}
We will now show that the reference solution $u$ is unique in the class of admissible functions $v\in V_T$.
Given $u_0\in H^2(M,\bar g)$, let $u\in U_T$ be the unique, global reference solution of the normalised Ricci flow provided by \autoref{Struwe} and let $v\in V_T$ be an admissible weak solution of the normalised Ricci flow on $[0,T]$ with the same initial condition $v(0)=u(0)=u_0\in H^2(M,\bar g)$. Let $w=u-v$ and $\partial_tw$ be the weak derivative of $w$ with respect to $t$ (which exists, since $u$ and $v$ both have a weak time derivative).

Since we know that $u$ is a strong solution of the Ricci flow, $u$ satisfies the equation
\begin{equation}\label{nochmals1}
\partial_tu(t)=\e^{-2u(t)}\Delta_{\bar g}u(t)+\bar K(1-\e^{-2u(t)}),\:t>0;\quad u(0)=u_0\in H^2(M,\bar g)
\end{equation}
pointwise almost everywhere.
Furthermore, for an admissible weak solution $v\in V_T$ of the Ricci flow on $[0,T]$ we have the relation~\eqref{weaksol}.

Using test functions of the form $\e^{-2u}\varphi$ with $\varphi\in C^\infty_c((0,T);C^\infty(M,\bar g))$ in \eqref{weaksol} (which are allowed since $\e^{-2u}\in L^\infty_tH^1_x$, see \autoref{uniform1}, and since all functions in $L^\infty_t H^1_x$ are admissible test functions, see \autoref{TF}) and subtracting this relation from the equation \eqref{nochmals1}, which we first multiply with $\varphi$ and integrate over $[0,T]\times M$, we obtain
\begin{multline}\label{Gleichungw}
\int_0^T\int_Me^{-2w(t)}\partial_tw(t)\varphi(t) d\mu_{\bar g}dt=-\int_0^T\int_M\langle\nabla_{\bar g}w(t),\nabla_{\bar g}(\e^{-2u(t)}\varphi(t))\rangle_{\bar g}d\mu_{\bar g}dt\\
+\int_0^T\int_M(\e^{-2w(t)}-1)\partial_tu(t)\varphi(t)d\mu_{\bar g}dt+\int_0^T\int_M \bar K(1-\e^{-2w(t)})\varphi(t)d\mu_{\bar g}dt
\end{multline}
for all test functions $\varphi\in C^\infty_c((0,T);C^\infty(M,\bar g))$.
We remark that the left hand side is well defined since $u\in U_T$ and
\begin{align*}
\int_0^T\int_M\e&^{-2w(t)}\partial_tw(t)\varphi(t)d\mu_{\bar g}dt\\
&=\int_0^T\int_M\e^{2v(t)}\e^{-2u(t)}\partial_tu(t)\varphi(t)d\mu_{\bar g}dt - \int_0^T\int_M\e^{2v(t)}\e^{-2u(t)}\partial_tv(t)\varphi(t)d\mu_{\bar g}dt\\
\end{align*}
with $\e^{v}\in L^\infty_tL^4_x$ and $\e^v\partial_t v\in L^2_tL^2_x$.
In fact, we can show (for details see \autoref{TF}) that the relation \eqref{Gleichungw} also holds for every $\varphi\in V_T$.

With these results we can now turn to the proof of our main result. Let $u\in U_T$ and $v\in V_T$ be as above and set $w=v-u\in V_T$. We have to show that $w\equiv0$ almost everywhere. We first show that $w\ge0$ almost everywhere on $[0,T]\times M$ for $T>0$ small enough. 

\begin{proposition}\label{wpositiv}
Given $u_0\in H^2(M,\bar g)$ let $u\in U_T$ be the reference solution of \eqref{konformRF} provided by \autoref{Struwe} and let $v\in V_T$ be an admissible weak solution of \eqref{konformRF} with initial data $v(0)=u_0=u(0)$. Then, if $T>0$ is sufficiently small, there holds $w=u-v\ge0$ almost everywhere on $[0,T]\times M$.
\end{proposition}

\begin{remark}
A similar idea was used by Giesen and Topping in \cite{GieTop10} to show uniqueness of the Ricci flow starting on a smooth surface of uniformly negative curvature with a possibly incomplete initial metric such that the smooth flows become instantaneously complete.
\end{remark}

\begin{proof}
As we have seen above $w$ fulfils the relation \eqref{Gleichungw} for all test function $\varphi\in V_T$. 

By a result due to Stampacchia (see e.g.~\cite{Sta65}) for every function $f\in H^1(M,\bar g)$ also $f_-:=\min\{f,0\}$ is in $H^1(M,\bar g)$ and we have
\begin{equation}\label{w_-}
\nabla_{\bar g} f_-=
\begin{cases}
\phantom{-}0\quad&\text{almost everywhere on}\quad \{f\ge0\};\\
\nabla_{\bar g}f\quad&\text{almost everywhere on}\quad \{f<0\},
\end{cases}
\end{equation}
and similarly for $f\in W^{1,1}_{t,x}$.
With $w=u-v\in V_T$, we thus have $w_-\in V_T$ and therefore
\begin{multline}\label{Test}
\int_0^T\int_M\e^{-2w(t)}\partial_tw(t)w_-(t) d\mu_{\bar g}dt=-\int_0^T\int_M\langle\nabla_{\bar g}w(t),\nabla_{\bar g}(\e^{-2u(t)}w_-(t))\rangle_{\bar g}d\mu_{\bar g}dt\\
+\int_0^T\int_Mw_-(\e^{-2w(t)}-1)\partial_tu(t)d\mu_{\bar g}dt+\bar K\int_0^T\int_Mw_-(t)(1-\e^{-2w(t)})d\mu_{\bar g}dt.
\end{multline}

Since $w$, $\partial_tw\in L^1_tL^1_x$ and therefore $w\in W^{1,1}_tL^1_x$, we get that $w_-\in W^{1,1}_tL^1_x$ with
\begin{equation}
\partial_t(w_-(t))=
\begin{cases}
\phantom{-}0\quad&\text{almost everywhere on}\quad \{w\ge0\};\\
\partial_tw(t)\quad&\text{almost everywhere on}\quad \{w<0\}.
\end{cases}
\end{equation}
So we have $\partial_t w(t)=\partial_tw_-(t)+\partial_tw_+(t)$ almost everywhere with $w_+:=\max\{0,w\}$. Let
\begin{equation}\label{F}
F:(-\infty,\infty)\to \R;\quad \xi\mapsto F(\xi):=\int_0^\xi\eta\e^{-2\eta}d\eta=\frac14(1-\e^{-2\xi}(2\xi+1))\ge0.
\end{equation}

By the definition of $w_-$ and the chain rule of Sobolev functions we have
\begin{equation}\label{w-minus}
\begin{split}
\int_0^T\int_M&\e^{-2w(t)}\partial_tw(t)w_-(t)d\mu_{\bar g}dt
=\int_0^T\int_M\e^{-2w_-(t)}\partial_t(w_-(t))w_-(t)d\mu_{\bar g}dt \\
&\phantom{aa}=\int_0^T\int_M\frac{d}{dt}F(w_-(t))d\mu_{\bar g}dt=\int_MF(w_-(t))d\mu_{\bar g}\bigg|_0^T=\int_MF(w_-(T))d\mu_{\bar g},
\end{split}
\end{equation}
where we used that $w(0)=0$.

We insert \eqref{w-minus} into \eqref{Test} to obtain
\begin{multline}\label{Ab}
\int_MF(w_-(T))d\mu_{\bar g}+\int_0^T\int_M\e^{-2u(t)}|\nabla_{\bar g}w_-(t)|^2_{\bar g}d\mu_{\bar g}dt\\
\begin{multlined}[t]
=2\int_0^{T}\int_M\e^{-2u(t)}\langle\nabla_{\bar g}u(t),\nabla_{\bar g}w_-(t)\rangle_{\bar g}w_-(t)d\mu_{\bar g}dt\\
+\bar K\int_0^{T}\int_Mw_-(t)(1-\e^{-2w_-(t)})d\mu_{\bar g}dt\\
+\int_0^{T}\int_Mw_-(t)(\e^{-2w_-(t)}-1)\partial_tu(t)d\mu_{\bar g}dt.
\end{multlined}
\end{multline}

By Young's inequality $2ab\le\frac{a^2}{2}+2b^2$ we have
\begin{align*}
2\int_0^{T}\int_M&\e^{-2u(t)}\langle\nabla_{\bar g}u(t),\nabla_{\bar g}w_-(t)\rangle_{\bar g}w_-(t)d\mu_{\bar g}dt\\
&\le\frac12\int_0^T\int_M\e^{-2u(t)}|\nabla_{\bar g}w_-(t)|_{\bar g}^2d\mu_{\bar g}dt+2\int_0^T\int_M\e^{-2u(t)}|\nabla_{\bar g}u(t)|^2_{\bar g}|w_-(t)|^2d\mu_{\bar g}dt.
\end{align*}

Thus from \eqref{Ab} we arrive at
\begin{equation}\label{HG}
\int_MF(w_-(T))d\mu_{\bar g}+\frac12\int_0^T\int_M\e^{-2u(t)}|\nabla_{\bar g}w_-(t)|^2_{\bar g}d\mu_{\bar g}dt\le A(T)+B(T)+C(T),
\end{equation}
where
\begin{align*}
A(T)&:=2\int_0^{T}\int_M\e^{-2u(t)}|\nabla_{\bar g}u(t)|^2_{\bar g}|w_-(t)|^2d\mu_{\bar g}dt,\\
B(T)&:=|\bar K|\int_0^{T}\int_M|w_-(t)(1-\e^{-2w_-(t)})|d\mu_{\bar g}dt,\\
C(T)&:=\int_0^{T}\int_M|w_-(t)(\e^{-2w_-(t)}-1)\partial_tu(t)|d\mu_{\bar g}dt.
\end{align*}

We define 
\begin{equation}\label{psi}
\psi(t):=\int_M F(w_-(t))d\mu_{\bar g}.
\end{equation}
In equation \eqref{HG} we can replace $T$ by an arbitrary $t\in[0,T]$. Using that
\[\frac12\int_0^T\int_M\e^{-2u(t)}|\nabla_{\bar g}w_-(t)|^2_{\bar g}d\mu_{\bar g}dt\ge0\]
we then obtain
\[\psi(t)\le A(t)+B(t)+C(t).\]
Taking now the essential supremum on both sides and using that $A$, $B$ and $C$ are non-decreasing in $t$, we get
\begin{equation}\label{E1}
\|\psi\|_{L^\infty([0,T])}\le A(T)+B(T)+C(T).
\end{equation}

On the other hand, since $F(\xi)\ge0$ for all $\xi\in\R$, from equation \eqref{HG} we also have \begin{equation}\label{E3}
\frac12 \int_0^T\int_M\e^{-2u(t)}|\nabla_{\bar g}w_-(t)|^2_{\bar g}d\mu_{\bar g}dt\le A(T)+B(T)+C(T).
\end{equation}
The bounds \eqref{E1} and \eqref{E3} yield
\[\|\psi\|_{L^\infty([0,T])}+\frac12 \int_0^T\int_M\e^{-2u(t)}|\nabla_{\bar g}w_-(t)|^2_{\bar g}d\mu_{\bar g}dt\le 2(A(T)+B(T)+C(T)).\]
Since $u\in L^\infty_tL^\infty_x$ there exists a uniform lower bound $\e^{-2u}\ge C_1>0$.
Moreover, we have
\[F(\xi)=\frac14(1-\e^{-2\xi}(2\xi+1))\ge\frac12|\xi|^2\quad\text{for all}\quad \xi\le0.\]
We may assume that $T\le1$. With
\[0<C_2:=\frac{2}{\min\{1,C_1\}}<\infty\]
and using the estimates for $A(T)$, $B(T)$ and $C(T)$ proved in Lemmas \ref{A1} - \ref{A3} below we thus find
\begin{equation}
\begin{split}\label{HG2}
\|w_-\|^2_{L^\infty_tL^2_x}+\|\nabla_{\bar g}w_-\|^2_{L^2_tL^2_x}&\le C_2\left(\|\psi\|_{L^\infty([0,T])}+\frac12\|\nabla_{\bar g}w_-\|^2_{L^2_tL^2_x}\right)\\
&\le 2C_2(A(T)+B(T)+C(T)). \\
&\le 2C_2(\delta_A(T)+|\bar K|\delta_B(T)+\delta_C(T))(\|w_-\|^2_{L^\infty_tL^2_x}+\|\nabla_{\bar g}w_-\|^2_{L^2_tL^2_x})\\
&=:\delta(T)(\|w_-\|^2_{L^\infty_tL^2_x}+\|\nabla_{\bar g}w_-\|^2_{L^2_tL^2_x}),
\end{split}
\end{equation}
where $\delta(T):=2C_2(\delta_A(T)+|\bar K|\delta_B(T)+\delta_C(T))$.
Replacing $T$ by a smaller number $T_1>0$, if necessary, by Lemmas \ref{A1} - \ref{A3} we may assume that $T=T_1\le 1$ and
\[\delta(T):=2C_2(\delta_A(T)+|\bar K|\delta_B(T)+\delta_C(T))<1.\]
So, we see that $w_-\equiv 0$ almost everywhere on $[0,T]\times M$. 
\end{proof}

\begin{proof}[Proof of \autoref{HT}]
First we show that $w\equiv 0$ almost everywhere on $[0,T]\times M$ if $T>0$ is as in \autoref{wpositiv} so that $w\ge0$ almost everywhere on $[0,T]\times M$.
Assume by contradiction that $w\ne0$. Then there exists an open set $U_0\subset [0,T]\times M$ of positive measure $\int_{U_0}d\mu_{\bar g}dt>0$ such that
\[w(t,x)>0\quad\text{for almost all}\quad (t,x)\in U_0.\]
But then
\[\int_0^T\int_M(\e^{2u(t)}-\e^{2v(t)})d\mu_{\bar g}dt\ge\int_{U_0}(\e^{2w(t)}-1)\e^{2v(t)}d\mu_{\bar g}dt>0,\]
which contradicts \autoref{VolumenerhaltungV}.

Therefore there is no subset of $[0,T]\times M$ with positive measure where $w>0$ almost everywhere. It follows that $w\equiv0$ almost everywhere on $[0,T]\times M$. 

Let now $T>0$ be arbitrary. Using the Sobolev embedding $W^{1,1}_tL^1_x\hookrightarrow C^0_tL^1_x$ we observe that the set
\[I:=\{t\ge0\mid w\equiv0 \text{ almost everywhere on }[0,t]\times M\}\]
is closed. Starting the flow at any time $t_0\in I$ with initial condition $w(t_0,\cdot)=0$ by the argument above there is an $\varepsilon>0$ such that $w\equiv0$ almost everywhere on $[t_0,t_0+\varepsilon]\times M$. So, $I$ is open. With $0\in I$ we see that $I\neq\emptyset$. Therefore $u\equiv v$ almost everywhere on $[0,T]\times M$, which concludes the proof.
\end{proof}

\section{Estimation of Integrals}\label{AI}
For the same notion as in the proof of \autoref{wpositiv} let
\begin{align*}
A(T)&:=2\int_0^{T}\int_M\e^{-2u(t)}|\nabla_{\bar g}u(t)|^2_{\bar g}|w_-(t)|^2d\mu_{\bar g}dt,\\
B(T)&:=\bar K\int_0^{T}\int_M|w_-(t)(1-\e^{-2w_-(t)})|d\mu_{\bar g}dt,
\end{align*}
and
\begin{align*}
C(T):=\int_0^{T}\int_M|w_-(t)(\e^{-2w_-(t)}-1)\partial_tu(t)|d\mu_{\bar g}dt.
\end{align*}
In the following three lemmas we show that all three integrals $A(T)$, $B(T)$ and $C(T)$ can be estimated by an arbitrary small multiple of $\|w_-\|_{L^\infty_tL^2_x}^2+\|\nabla_{\bar g}w_-\|_{L^2_tL^2_x}^2$.
\begin{lemma}\label{A1}
We have
\begin{equation}\label{G1}
A(T)\le \delta_A(T)(\|w_-\|^2_{L^\infty_tL^2_x}+\|\nabla_{\bar g}w_-\|^2_{L^2_tL^2_x}),
\end{equation}
where $\delta_A(T)\to0$ for $T\downarrow0$.
\end{lemma}

\begin{proof}
By using that $u\in L^\infty_tL^\infty_x$ we get with H\"older's inequality 
\begin{align*}
A(T)&=2\int_0^{T}\int_M\e^{-2u(t)}|\nabla_{\bar g}u(t)|^2_{\bar g}|w_-(t)|^2d\mu_{\bar g}dt\\
&\le C\int_0^T\int_M|\nabla_{\bar g}u(t)|^2_{\bar g}|w_-(t)|^2d\mu_{\bar g}dt\\
&\le C\left(\int_0^T\int_M|\nabla_{\bar g}u(t)|^4_{\bar g}d\mu_{\bar g}dt\right)^{\frac12}\left(\int_0^T\int_M|w_-(t)|^4d\mu_{\bar g}dt\right)^{\frac12}.\\
\end{align*}
So with Sobolev's inequality \eqref{Sobolev} we get 
\[A(T)\le \delta_A(T)(\|w_-\|^2_{L^\infty_tL^2_x}+\|\nabla_{\bar g}w_-\|^2_{L^2_tL^2_x}),\]
where (with a different constant $C$)
\[\delta_A(T)=C\left(\int_0^T\int_M|\nabla_{\bar g} u(t)|^4_{\bar g}d\mu_{\bar g}dt\right)^{\frac12}\]
tends to zero for $T\downarrow0$ since $\nabla_{\bar g}u\in L^\infty_tH^1_x\hookrightarrow L^4_tL^4_x$ by \autoref{SobolevLemma}.
\end{proof}

\begin{lemma}\label{A2}
We have
\begin{equation}\label{G2}
B(T)\le \delta_B(T)|\bar K|(\|w_-\|^2_{L^\infty_tL^2_x}+\|\nabla_{\bar g}w_-\|^2_{L^2_tL^2_x}),
\end{equation}
where $\delta_B(T)\to0$ for $T\downarrow0$.
\end{lemma}

\begin{proof}
We use the estimate
\begin{equation}\label{exp}
|1-\e^{-2w_-}|\le 2|w_-|\sum_{k\ge0}\frac{(2|w_-|)^k}{(k+1)!}\le2|w_-|\e^{2|w_-|}
\end{equation}
and again H\"older's inequality followed by Sobolev's inequality \eqref{Sobolev} to get
\begin{align*}
B(T)&=|\bar K|\int_0^T\int_M|w_-(t)(1-\e^{-2w_-(t)})|d\mu_{\bar g}dt\\
&\le 2|\bar K|\int_0^T\int_M|w_-(t)|^2\e^{2|w_-(t)|}d\mu_{\bar g}dt\\
&\le 2|\bar K|\left(\int_0^T\int_M|w_-(t)|^4d\mu_{\bar g}\right)^{\frac12}\left(\int_0^T\int_M\e^{4|w_-(t)|}d\mu_{\bar g}dt\right)^{\frac12}\\
&\le \delta_B(T)|\bar K|(\|w_-\|^2_{L^\infty_tL^2_x}+\|\nabla_{\bar g}w_-\|^2_{L^2_tL^2_x}),
\end{align*}
where 
\[\delta_B(T)=C\left(\int_0^T\int_M\e^{4|w_-(t)|}d\mu_{\bar g}dt\right)^{\frac12}.\]
So, also $\delta_B(T)$ tends to zero for $T\downarrow0$ since $w_-\in V_T$ and therefore with \autoref{Reg_w} $\e^{\pm w_-}\in L^\infty_tL^p_x$ for all $p\in[1,\infty)$.
\end{proof}

\begin{lemma}\label{A3}
We have
\begin{equation}\label{G3}
C(T)\le\delta_C(T)(\|w_-\|^2_{L^\infty_tL^2_x}+\|\nabla_{\bar g}w_-\|^2_{L^2_tL^2_x}),
\end{equation}
where $\delta_C(T)\to0$ for $T\downarrow0$.
\end{lemma}

\begin{proof}
Since $u\in L^\infty_tH^2_x\cap L^2_tH^3_x\cap H^1_tH^1_x$ and therefore $\partial_t u\in L^2_tH^1_x$ we see by using \autoref{kleineReg} that $\partial_tu\in L^2_tL^4_x$.
So, using \eqref{exp} we can estimate
\begin{align*}
C(T)&=\int_0^T\int_M|w_-(t)(\e^{-2w_-(t)}-1)\partial_tu(t)|d\mu_{\bar g}dt\\
&\le 2\int_0^T\int_M|w_-(t)|^2\e^{2|w_-(t)|}|\partial_tu(t)|d\mu_{\bar g}dt\\
&\le 2\left(\int_0^T\int_M|w_-(t)|^4d\mu_{\bar g}dt\right)^{\frac12}\left(\int_0^T\int_M\e^{4|w_-(t)|}|\partial_tu(t)|^2d\mu_{\bar g}dt\right)^{\frac12}.
\end{align*}
With
\[\int_0^T\int_M\e^{4|w_-(t)|}|\partial_tu(t)|^2d\mu_{\bar g}dt\le\int_0^T\left[\left(\int_M\e^{8|w_-(t)|}d\mu_{\bar g}\right)^{\frac12}\left(\int_M|\partial_tu(t)|^4d\mu_{\bar g}\right)^{\frac12}\right]dt\]
we finally get
\[C(T)\le 2\|\e^{|w_-|}\|^2_{L^\infty_tL^8_x}\|\partial_tu\|_{L^2_tL^4_x}\|w_-\|^2_{L^4_tL^4_x}\le \delta_C(T)(\|w_-\|^2_{L^\infty_tL^2_x}+\|\nabla_{\bar g}w_-\|^2_{L^2_tL^2_x}),\]
where
\[\delta_C(T)=C\|\e^{|w_-|}\|^2_{L^\infty_tL^8_x}\|\partial_tu\|_{L^2_tL^4_x}.\]
Since $\partial_tu\in L^2_tL^4_x$, and $\e^{\pm w_-}\in L^\infty_tL^p_x$ for all $p\in[1,\infty)$ by \autoref{Reg_w} we see that $\delta_C(T)$ tends to zero for $T\downarrow0$.
\end{proof}

\appendix
\section{Appendix}

\subsection{Regularity Results}\label{RR}
For the proof of our main theorem, we list some useful properties of $H^1$-functions 
and some further regularity results.

\begin{lemma}[Gagliardo--Nirenberg inequality, \cite{CecMon08}]\label{eindSob}
There exists a constant $C=C(M,\bar g)$ such that we have for every $f\in H^1(M,\bar g)$ the inequality
\[\|f\|^4_{L^4(M,\bar g)}
\le C\|f\|^2_{L^2(M,\bar g)}\|f\|^2_{H^1(M,\bar g)}.\]
\end{lemma}

\begin{remark}\label{kleineReg}
With \autoref{eindSob} we therefore have $L^p_tH^1_x\subset L^p_tL^4_x$ for all $p\in[1,\infty]$.
\end{remark}

\begin{lemma}[Sobolev inequality]\label{SobolevLemma}
There exists a constant $C>0$ such that for every $f\in V_T$, $T\le1$, we have
\begin{equation}\label{Sobolev}
\|f\|^2_{L^4_tL^4_x}\le C(\|f\|^2_{L^\infty_tL^2_x}+\|\nabla_{\bar g}f\|^2_{L^2_tL^2_x})<\infty.
\end{equation}
\end{lemma}

\begin{proof}
With \autoref{eindSob} we have for all $T\le1$ that
\begin{align*}
\|f\|^4_{L^4_tL^4_x}&=\int_0^T\|f(t)\|^4_{L^4(M,\bar g)}dt \le C\int_0^T\|f(t)\|^2_{L^2(M,\bar g)}\|f(t)\|^2_{H^1(M,\bar g)}dt\\
&\le C\|f\|^2_{L^\infty_tL^2_x}\int_0^T(\|f(t)\|_{L^2(M,\bar g)}^2+\|\nabla_{\bar g}f(t)\|^2_{L^2(M,\bar g)})dt\\
&\le C\cdot T\:\|f\|^4_{L^\infty_tL^2_x}+C\|f\|^2_{L^\infty_tL^2_x}\|\nabla_{\bar g}f\|^2_{L^2_tL^2_x}\\
&\le C\left(\|f\|^4_{L^\infty_tL^2_x}+\|f\|^2_{L^\infty_tL^2_x}\|\nabla_{\bar g}f\|^2_{L^2_tL^2_x}\right).
\end{align*}
By using Young's inequality we have
\[\|f\|_{L^\infty_tL^2_x}\|\nabla_{\bar g}f\|_{L^2_tL^2_x}\le \frac12\left(\|f\|^2_{L^\infty_tL^2_x}+\|\nabla_{\bar g}f\|^2_{L^2_tL^2_x}\right)\]
and therefore
\[\|f\|^2_{L^4_tL^4_x}
\le C(\|f\|^2_{L^\infty_tL^2_x}+\|\nabla_{\bar g}f\|^2_{L^2_tL^2_x})\]
and this is finite since $f\in V_T\subset L^p_tH^1_x$ for all $p<\infty$ and any $T\le1$.
\end{proof}

\begin{lemma}\label{uniform1}
We have the embedding $U_T\hookrightarrow L^\infty_tL^\infty_x$.
\end{lemma}

\begin{proof}
By using the Sobolev embedding theorem and the fact that $\vol_{\bar g}<\infty$, we have that $H^2(M,\bar g)\hookrightarrow C^0(M,\bar g)\hookrightarrow L^\infty(M,\bar g)$. So we get $L^\infty_tH^2_x\hookrightarrow L^\infty_tL^\infty_x$. In particular we see that every function $u\in U_T$ is in $L^\infty_tL^\infty_x$.
\end{proof}

\begin{remark}\label{uniform2}
Later we will use this embedding for our reference solution $u\in U_T$. Of course we also have $C^\infty_tC^\infty_x\hookrightarrow L^\infty_tL^\infty_x$. So the same embedding is true also for the reference solution provided by Hamilton.
\end{remark}

For the last regularity result we need the following lemma (see e.g.~\cite[Corollary 1.7]{Cha04}) which is a consequence of the Trudinger--Moser inequality:

\begin{lemma}\label{Onofri-ungleichung}
For a two-dimensional, closed Riemannian manifold $(M,\bar g)$ there are constants $\eta>0$ and $C=C(\bar g)>0$ such that
\begin{equation}\label{OU}
\int_M\e^{(f-\bar f)}d\mu_{\bar g}\le C\exp\left(\eta\|\nabla_{\bar g}f\|^2_{L^2(M,\bar g)}\right)
\end{equation}
for all $f\in H^1(M,\bar g)$ where
\[\bar f:=\frac{1}{\vol(M,\bar g)}\int_Mf\:d\mu_{\bar g}=\int_Mf\:d\mu_{\bar g},\]
in view of our assumption that $\vol_{\bar g}=1$.
\end{lemma}

With \autoref{Onofri-ungleichung} we can show the following result.

\begin{corollary}\label{Reg_w}
For $f\in V_T$ we have 
\begin{equation}
\e^{\pm f}\in L^\infty_tL^p_x \quad\text{for all}\quad p\in[1,\infty).
\end{equation}
\end{corollary}

\begin{proof}
We know that $pf\in V_T$ for all $p\in[1,\infty)$ and so we get with \autoref{Onofri-ungleichung}
\begin{align*}
\|\e^f\|^p_{L^\infty_tL^p_x}&=\underset{t\in[0,T]}{\esssup}\int_M\e^{pf(t)}d\mu_{\bar g}\\
&=\underset{t\in[0,T]}{\esssup}\int_M\e^{pf(t)- p\bar f(t)+ p\bar f(t)}\\
&\le \underset{t\in[0,T]}{\esssup}\left(C\exp(\eta\|\nabla_{\bar g}(pf(t))\|^2_{L^2(M,\bar g)})\:\e^{|p\bar f(t)|}\right)\\
&\le \underset{t\in[0,T]}{\esssup}\left(C\e^{\eta p^2\|\nabla_{\bar g}f(t)\|^2_{L^2(M,\bar g)}+p\|f(t)\|_{L^1(M,\bar g)}}\right).
\end{align*}
This expression is finite since $f\in V_T$ and, using that $\vol_{\bar g}<\infty$, therefore also $f\in L^\infty_tL^1_x$. So we have $\e^f\in L^\infty_tL^p_x$ for all $p\in [1,\infty)$. Replacing $f$ by $-f$ we similarly get that $\e^{-f}\in L^\infty_tL^p_x$ for all $p\in[1,\infty)$.
\end{proof}

\begin{remark}
Since $T<\infty$ and $\vol_{\bar g}<\infty$, we also have for every $p\in[0,\infty)$ the embedding $L^\infty_tL^p_x\hookrightarrow L^q_tL^r_x$ for $q<\infty$, $r\le p$.
\end{remark}

\subsection{Allowed Test Functions}\label{TF}
Functions $\psi\in L^\infty_tH^1_x=V_T$ are allowed as test functions in \eqref{Gleichungw}.
To see this let $w=u-v$ as in \autoref{proof}. Recall that we have
\begin{multline*}
\int_0^T\int_Me^{-2w(t)}\partial_tw(t)\varphi(t) d\mu_{\bar g}dt=-\int_0^T\int_M\langle\nabla_{\bar g}w(t),\nabla_{\bar g}(\e^{-2u(t)}\varphi(t))\rangle_{\bar g}d\mu_{\bar g}dt\\
+\int_0^T\int_M(\e^{-2w(t)}-1)\partial_tu(t)\varphi(t)d\mu_{\bar g}dt+\int_0^T\int_M \bar K(1-\e^{-2w(t)})\varphi(t)d\mu_{\bar g}dt
\end{multline*}
for all test functions $\varphi\in C^\infty_c((0,T),C^\infty(M,\bar g))$.
By density of smooth functions in $L^2_tH^1_x$, for any $\psi\in V_T\subset L^2_tH^1_x$ (recall that $T<\infty$) we find a sequence $(\varphi_n)\subset C^\infty_c((0,T),C^\infty(M,\bar g))$ such that
\[\|\varphi_n-\psi\|_{L^2_tH^1_x}\to 0\quad\text{for}\quad n\to\infty.\]
Plugging in $\varphi_n$ instead of $\varphi$ and using again H\"older's inequality several times as well as Young's inequality, we see that we may pass to the limit $n\to\infty$.
So we see that $\psi\in L^\infty_tH^1_x$ is an allowed test function for the equation above.

\printbibliography

\end{document}